\documentclass[12pt,reqno]{amsart}

\usepackage{amssymb,verbatim}
\usepackage{amsmath,amsfonts}
\usepackage[mathscr]{euscript} 
\usepackage{amsthm}
\usepackage{color}
\usepackage{url}
\usepackage{graphicx} 
\usepackage{enumitem} 
\usepackage{lineno,xcolor} 
\usepackage{hyperref} 
\hypersetup{colorlinks} 
\usepackage{bbm} 
\usepackage{mathrsfs} 

\usepackage[lined,algonl,boxed,norelsize]{algorithm2e} 
\usepackage{diagbox}
\usepackage{threeparttable} 
\usepackage{extarrows} 
\usepackage{tikz-cd} 
\usepackage{ltablex} 

\usepackage{graphicx} 

\usepackage{standalone} 

\setlist[enumerate]{
label=\textnormal{({\roman*})},
ref={\roman*}}

\makeatletter
\def\th@plain{%
  \thm@notefont{}
  \itshape 
}
\def\th@definition{%
  \thm@notefont{}
  \normalfont 
}
\makeatother

\makeatletter
\newtheorem*{rep@theorem}{\rep@title}
\newcommand{\newreptheorem}[2]{%
\newenvironment{rep#1}[1]{%
 \def\rep@title{#2 \ref{##1}}%
 \begin{rep@theorem}}%
 {\end{rep@theorem}}}
\makeatother

\newreptheorem{theorem}{Theorem}
\newreptheorem{corollary}{Corollary}

\newtheorem{theorem}{Theorem}
\newtheorem{lemma}{Lemma}[section]
\newtheorem{proposition}[lemma]{Proposition}

\theoremstyle{remark}

\theoremstyle{definition}

\makeatletter
\newcommand*{\da@rightarrow}{\mathchar"0\hexnumber@\symAMSa 4B }
\newcommand*{\da@leftarrow}{\mathchar"0\hexnumber@\symAMSa 4C }
\newcommand*{\xdashrightarrow}[2][]{%
  \mathrel{%
    \mathpalette{\da@xarrow{#1}{#2}{}\da@rightarrow{\,}{}}{}%
  }%
}
\newcommand{\xdashleftarrow}[2][]{%
  \mathrel{%
    \mathpalette{\da@xarrow{#1}{#2}\da@leftarrow{}{}{\,}}{}%
  }%
}
\newcommand*{\da@xarrow}[7]{%
  \sbox0{$\ifx#7\scriptstyle\scriptscriptstyle\else\scriptstyle\fi#5#1#6\m@th$}%
  \sbox2{$\ifx#7\scriptstyle\scriptscriptstyle\else\scriptstyle\fi#5#2#6\m@th$}%
  \sbox4{$#7\dabar@\m@th$}%
  \dimen@=\wd0 %
  \ifdim\wd2 >\dimen@
    \dimen@=\wd2 %
  \fi
  \count@=2 %
  \def\da@bars{\dabar@\dabar@}%
  \@whiledim\count@\wd4<\dimen@\do{%
    \advance\count@\@ne
    \expandafter\def\expandafter\da@bars\expandafter{%
      \da@bars
      \dabar@ 
    }%
  }%
  \mathrel{#3}%
  \mathrel{%
    \mathop{\da@bars}\limits
    \ifx\\#1\\%
    \else
      _{\copy0}%
    \fi
    \ifx\\#2\\%
    \else
      ^{\copy2}%
    \fi
  }%
  \mathrel{#4}%
}
\makeatother

    \makeatletter 
\newcommand{\overrightharpoon}{%
  \mathpalette{\overarrow@\rightharpoonfill@}}
\def\rightharpoonfill@{\arrowfill@\relbar\relbar\rightharpoonup}
\makeatother

\makeatletter 
\newcommand{\osh}{\mathpalette{\overarrowsmall@\rightharpoonfill@}}
\def\rightharpoonfill@{\arrowfill@\relbar\relbar\rightharpoonup}
\newcommand{\overarrowsmall@}[3]{%
  \vbox{%
    \ialign{%
      ##\crcr
      #1{\smaller@style{#2}}\crcr
      \noalign{\nointerlineskip}%
      $\m@th\hfil#2#3\hfil$\crcr
    }%
  }%
}
\def\smaller@style#1{%
  \ifx#1\displaystyle\scriptstyle\else
    \ifx#1\textstyle\scriptstyle\else
      \scriptscriptstyle
    \fi
  \fi
}
\makeatother

\makeatletter
\newcommand{\mylabel}[2]{#2\def\@currentlabel{#2}\label{#1}}
\makeatother

\DeclareMathOperator{\Ec}{\mathcal{E}} 
\DeclareMathOperator{\Gc}{\mathcal{G}} 
\DeclareMathOperator{\Zb}{\mathbb{Z}} 

\title{A rotor configuration with maximum escape rate}
\author{Swee Hong Chan}
 \thanks{Department of Mathematics, Cornell University. Partially supported by NSF grant DMS-1455272. Email: \url{sweehong@math.cornell.edu}.}

\begin{document}
 
 \begin{abstract}
 Rotor walk is a deterministic analogue of simple random walk.
 For any given graph,
 we construct a rotor configuration for which the escape rate of the corresponding rotor walk is equal to
 the escape rate of simple random walk,
 and thus answer a question of Florescu, Ganguly, Levine, and Peres (2014). 
 \end{abstract}



\maketitle
 
 \section{Introduction}\label{section: intro}
Let $G:=(V,E)$  be a graph that is connected, simple (i.e., no loops and multiple edges) and locally finite (i.e., every vertex has finitely many neighbors).
In a rotor walk~\cite{WLB96, PDDK96}, 
each vertex has a \emph{rotor}, which is an outgoing edge of the vertex. 
All of the rotors together constitute a \emph{rotor configuration}, which is encoded by a function $\rho$ that maps every vertex of $G$ to one of its outgoing edges.
To each vertex $x$ we  assign a fixed \emph{rotor mechanism}, which is a cyclic ordering on the set of outgoing edges $\Ec_x$ of $x$,
and is encoded by a bijection 
  $m_x: \Ec_x \to \Ec_x$  that has only one orbit.

 Rotor walk evolves in the following manner.
A {particle} is initially  located at a fixed  vertex $o$.
At each time step, 
the rotor at the particle's current location is first incremented to the next edge in the cyclic order, and the particle moves to the target vertex of the new rotor.


Propp~\cite{Propp03} proposed  rotor walk as a derandomized version of  simple random walk, and 
this naturally invited a comparison between the two walks.
One such comparison is given by  the following experiment.
Start with an  initial rotor configuration $\rho$, and with $n$ particles initially located at $o$.
At each time step, each of these $n$ particles will take turns in performing one step of rotor walk,
and  the particle is removed if it ever returns to $o$.
Denote by $I(\rho,n)$ the number of particles that never return to $o$.

%
Schramm~\cite{HP10, FGLP14} showed that the \emph{escape rate} of rotor walk is always bounded  above by the escape rate of  simple random walk. That is to say, for any rotor configuration $\rho$:
\begin{equation}\label{equation: Schramm bound}
\limsup_{n\to \infty} \frac{I(\rho,n)}{n} \leq \alpha_G,
\end{equation}
where $\alpha_G$ is the probability for  simple random walk starting at $o$ to never return to $o$.



The result of Schramm inspired Florescu, Ganguly, Levine, and Peres \cite{FGLP14} to ask  if there is always  a rotor configuration with  escape rate equal to $\alpha_G$.
 Such a configuration has been constructed  for certain choices of  $G$,
such as   for the
binary tree~\cite{LL09};  for  transient trees~\cite{AH11};
for $\Zb^d$ with $d\geq 3$~\cite{He14}; and  for  transient vertex-transitive graphs~\cite{Chan18}.

In this paper, we resolve  the question of Florescu et al.\ by constructing 
 a rotor configuration  with maximum escape rate for any given graph.  
We    focus on the case when $G$ is a transient graph,
as any rotor configuration on a recurrent graph has   escape rate equal to $0$ by \eqref{equation: Schramm bound}.

Let $\Gc:V \to \mathbb{R}_{\geq 0}$ be the \emph{Green function} of $G$,
which maps $x\in V$ to the expected number of visits to $x$ by the simple random walk on $G$ starting at $o$.
The
 \emph{weight} of a directed edge $(x,y)$ of $G$ is 
\begin{equation}\label{equation: weight}
w(x,y):= \frac{-1}{\deg(x)}\sum_{i=0}^{\deg(x)-1} i \frac{\Gc(y_{i+1})}{\deg(y_{i+1})},
\end{equation}  
where  $(x,y_i):=m_x^i(x,y)$ is the edge obtained by incrementing the edge $(x,y)$ for $i$ consecutive times by using the rotor mechanism at $x$.


\begin{theorem}\label{theorem: yes we can}
Let $G$ be a transient graph that is connected, simple, and locally finite.
If $\rho_{\min}$ is a rotor configuration such that, for any vertex $x$ and any outgoing edge $(x,y)$  of $x$, 
\begin{equation}\label{equation: minimizer condition}
 w(\rho_{\min}(x))\leq  w(x,y),
 \end{equation}
then 
\begin{equation*}
 \lim_{n \to \infty} \frac{I(\rho_{\min},n)}{n}=\alpha_G. 
 \end{equation*}
\end{theorem}
 Theorem~\ref{theorem: yes we can} is proved by constructing an invariant of the rotor walk that balances between the Green function of the location of the particles and the weight of the edges in the rotor configuration at any given time.

Note that one can always construct  a rotor configuration $\rho$ satisfying \eqref{equation: minimizer condition}, by defining    $\rho(x)$ for any $x \in V$ to be the edge  for which its weight is the minimum among all outgoing edges of $x$.
Also note that \eqref{equation: minimizer condition} is not a necessary condition, as all   configurations with maximum escape rate from other works (mentioned above)   do not satisfy \eqref{equation: minimizer condition}.



\section{Proof of Theorem~\ref{theorem: yes we can}}

We now give  a formal definition to the experiment in Section~\ref{section: intro}.
Let $\rho$ be the initial rotor configuration, and let $n$ be the number of particles.
The location of the particles 
$X_t^{(0)}, X_t^{(1)}, \ldots, X_t^{(n-1)}$ and the rotor configuration $\rho_t$
at the $t$-th step of the experiment $(t\geq 0)$ are given by the following recurrence:
%
\begin{itemize}
\item Initially, $X_{0}^{(i)}:=o$ for  $i \in \{0,1,\ldots,n-1\}$ and $\rho_0:=\rho$;
\item Write $i_t:= t+1 \ \text{mod}\ n$. If the $i_t$-th particle has returned to $o$ (i.e. $X_t^{(i_t)}=o$ and $X_s^{(i_t)}\neq o$ for some $s<t$), then 
\[\rho_{t+1}:=\rho_t, \quad \text{and} \quad X_{t+1}^{(i)}:=X_t^{(i)} \ \text{ for  $i \in \{0,\ldots,n-1\}$}.  \]

\item If the $i_t$-th particle has not returned to $o$, then 
\begin{align*}
 \rho_{t+1}(x)&:=\begin{cases}m_x(\rho_t(x)) & \text{ if } x=X_t^{(i_t)};\\
\rho_t(x) & \text{ otherwise.}
\end{cases};  \\
X_{t+1}^{(i)}&:= \begin{cases}
\text{target vertex of } \rho_{t+1}(X_t^{(i)}) & \text{ if } i=i_t;\\
X_t^{(i)} & \text{ otherwise.}
\end{cases}
\end{align*}
\end{itemize}
That is, at time $t$, the $i_t$-th particle performs one step of a rotor walk if it has not returned to $o$, and  does nothing  if it has returned to $o$.

We denote by $R_t:=R_t(\rho,n)$ the \emph{range} of the experiment at time $t$,
\[ R_t:=\{X_{s}^{(i)} \mid i \in   \{0,1,\ldots,n-1\} \text{ and } s  \leq t  \}. \]

We now define an invariant of the rotor walk
that is a a special case of the invariant introduced in \cite[Proposition~13]{HP10}; a related invariant has been used in  \cite{HS11} and \cite{HS12} to study the rotor-router aggregation of  comb lattices.
Let  $M_t:=M_t(\rho,n)$ ($t\geq 0$) be given by:
\begin{equation}\label{equation: M_t}
  M_t:= \sum_{i=0}^{n-1} \frac{\Gc(X_t^{(i)})}{\deg(X_t^{(i)})} +\frac{\min \{t,n\} }{\deg(o)} +\sum_{x \in R_t}  \left(w(\rho_t(x))-w(\rho(x))\right). 
\end{equation}


\begin{proposition}\label{proposition: martingale}
For any initial rotor configuration $\rho$, any $n\geq 1$, and any $t\geq 0$,
we have 
\[ M_t= n \frac{\Gc(o)}{\deg(o)}.  \]
\end{proposition}

We will use the fact that the Green function is a voltage function when a unit current enters $G$ through $o$  \cite[Proposition~2.1]{LP16}. That is, for any $x \in V$, 
\begin{align}\label{equation: dirichlet problem}
\frac{1}{\deg(x)}\sum_{y \sim x} \frac{\Gc(y)}{\deg(y)}= \frac{\Gc(x)}{\deg(x)}-\frac{\mathbbm{1}\{x=o\}}{\deg(o)},
\end{align}
where $y \sim x$ means that $y$ is a neighbor of $x$ in $G$.

\begin{proof}[Proof of Proposition~\ref{proposition: martingale}]
It follows directly from the definition that $M_0= n\frac{\Gc(o)}{\deg(o)}$.
Therefore it suffices to show that, for any $t\geq 0$:
\[ M_{t+1}-M_t=0.  \]

Recall that $i_t:= t+1 \text{ mod } n$.
Write
 $\alpha_t:=X_t^{(i_t)}$ and $\beta_{t}:=X_{t+1}^{(i_t)}$.
If the $i_t$-th particle has returned to $o$ by time $t$,
then no action is performed at time $t$, and $M_{t+1}=M_t$.
If the $i_t$-th particle has not returned to $o$ by time $t$,
then it follows  from the definition of $M_t$ and $M_{t+1}$ in \eqref{equation: M_t} that
\begin{equation}\label{equation: martingale 1}
\begin{split}
M_{t+1}-M_t=& \frac{\Gc(\beta_{t})}{\deg(\beta_{t})}-
 \frac{\Gc(\alpha_t)}{\deg(\alpha_t)}+ \frac{\mathbbm{1}\{t\leq n-1\}}{\deg(o)} \\
 &+w(\rho_{t+1}(\alpha_t))-w(\rho_{t}(\alpha_t)).
 \end{split}
\end{equation}
On the other hand, we have from the definition of $w$ in \eqref{equation: weight}  that
\begin{align*}
w(\rho_{t+1}(\alpha_t))-w(\rho_{t}(\alpha_t))=&\frac{-1}{\deg(\alpha_t)} \left( \deg(\alpha_t) \frac{\Gc(\beta_{t})}{\deg(\beta_{t})} -\sum_{y \sim \alpha_{t}} \frac{\Gc(y)}{\deg(y)}\right)\\
=&-\frac{\Gc(\beta_{t})}{\deg(\beta_{t})}+\frac{1}{\deg(\alpha_t)} \sum_{y \sim \alpha_{t}} \frac{\Gc(y)}{\deg(y)}.
\end{align*}
Applying  \eqref{equation: dirichlet problem} to the equation above then gives us
\begin{align}\label{equation: martingale 2}
w(\rho_{t+1}(\alpha_t))-w(\rho_{t}(\alpha_t))=&-\frac{\Gc(\beta_{t})}{\deg(\beta_{t})}+\frac{\Gc(\alpha_t)}{\deg(\alpha_t)}- \frac{\mathbbm{1}\{\alpha_t=o \}}{\deg(o)}.
\end{align}
Combining \eqref{equation: martingale 1} and \eqref{equation: martingale 2},
we then get
\[ M_{t+1}-M_t=\frac{\mathbbm{1}\{t\leq n-1\}}{\deg(o)}- \frac{\mathbbm{1}\{\alpha_t=o\}}{\deg(o)}.  \]

Since the $i_t$-th particle has not returned to $o$ yet by time $t$,
this means that the $\alpha_t=o$ if and only if $t\leq n-1$ (i.e., the $i_t$-th particle has not left $o$ yet).
This then implies $M_{t+1}-M_t=0$ by the equation above, and the proof is complete.
\end{proof}

Let $I_{t}(\rho,n)$ be the number of particles that have not returned to $o$ by time $t$.

\begin{proposition}\label{proposition: occupation measure}
If $\rho_{\min}$ is a configuration that satisfies \eqref{equation: minimizer condition},
then for   any $n\geq 1$ and any $t\geq n$,
\[ \frac{I_t(\rho_{\min},n)}{n}\geq   \alpha_G.  \]
\end{proposition}
\begin{proof}
Let $S_t \subseteq \{0,1\ldots,n-1\}$ be the set of particles that has returned to $o$ by time $t$.
Since the Green function is a nonnegative function, we have:
\begin{align}\label{equation: occupation 1}
  \sum_{i=0}^{n-1} \frac{\Gc(X_t^{(i)})}{\deg(X_t^{(i)})}\geq &  \sum_{i \in S_t} \frac{\Gc(X_t^{(i)})}{\deg(X_t^{(i)})}   =  \left(n-I_t(\rho_{\min},n) \right) \frac{\Gc(o)}{\deg(o)}.
\end{align}
Since $\rho_{\min}$ satisfies \eqref{equation: minimizer condition}, we also have 
\begin{equation}\label{equation: occupation 2}
\sum_{x \in R_t}  (w(\rho_t(x))-w(\rho_{\min}(x)))\geq 0. 
\end{equation}
Plugging \eqref{equation: occupation 1} and \eqref{equation: occupation 2} into the definition of $M_t$ in \eqref{equation: M_t}, we have that, for $t\geq n$,
\[ M_t\geq \left(n-I_t(\rho_{\min},n) \right) \frac{\Gc(o)}{\deg(o)} +\frac{n}{\deg(o)}, \]
which is equivalent to
\[ \frac{I_t(\rho_{\min},n)}{n}\geq  1+ \frac{1}{\Gc(o)}-\frac{\deg(o)}{n\Gc(o)} M_t.\]
Plugging in the value of $M_t$ from Proposition~\ref{proposition: martingale}, we then have:
\[ \frac{I_t(\rho_{\min},n)}{n}\geq   \frac{1}{\Gc(o)},\]
and the conclusion now follows by noting that $\alpha_{G}=1/\Gc(o)$.
%
\end{proof}

We now present the proof of Theorem~\ref{theorem: yes we can}.

\begin{proof}[Proof of Theorem~\ref{theorem: yes we can}]
We have for any $n\geq 1$,
\[ \frac{I(\rho_{\min},n)}{n}=\lim_{t\to \infty}  \frac{I_t(\rho_{\min},n)}{n}\geq \alpha_G, \]
where the inequality is due to Proposition~\ref{proposition: occupation measure}.
The theorem now follows by combining the inequality above with \eqref{equation: Schramm bound}.
\end{proof}

\section{Open problems}

\begin{enumerate}
\item Classify all $c\geq 0$ for which there exists a rotor configuration $\rho$ such that $\lim_{n \to \infty} I(\rho,n)/n=c$.

To date the only known result of this kind is due to Landau and Levine~\cite{LL09}, which shows that, for the complete binary tree, the constant $c$ can range from $0$ to $\alpha_G$.

\item Consider the random rotor configuration $\rho$ where $(\rho(x))_{x \in V}$ are  independent and uniformly distributed among the outgoing edges of $x$. 
What is the probability that $\rho$ has escape rate equal to $\alpha_G$?

Angel and Holroyd~\cite{AH11} showed that this probability is 1 if $G$  is a complete $b$-ary tree.
The  author~\cite{Chan18} also showed  the same result if $G$ is a transient vertex-transitive graph and  the  configuration is sampled  from the oriented wired spanning forest measure.
\end{enumerate}

\section*{Acknowledgement}
We would like to thank Lila Greco,  Dan Jerison, and Ecaterina Sava-Huss   for helpful comments on an earlier draft.

\bibliographystyle{alpha}
\bibliography{Schramm}

\end{document}